\documentclass[12pt, psamsfonts]{amsart}
\usepackage{amssymb,amsfonts,amsthm,amsmath,epsfig,hhline}

\usepackage{verbatim, enumerate}    %added
\usepackage{pgf,tikz}
\usetikzlibrary{arrows,fadings}

\usepackage[arc,all,graph,frame]{xy}

\theoremstyle{plain}
\newtheorem{theorem}{Theorem}
\newtheorem{corollary}{Corollary}

\theoremstyle{definition}
\newtheorem{definition}{Definition}

\newtheorem{remark}{Remark}

\newcommand{\N}{\mathbb{N}}
\newcommand{\R}{\mathbb{R}}
\newcommand{\Z}{\mathbb{Z}}

\def\id{\operatorname{id}}
\def\tr{\operatorname{tr}}

\newcommand{\drawaxes}{
\draw[color=black,line width =0.25pt] (-1.5,0) -- (1.5,0);
\foreach \x in {-1,1}
\draw[shift={(\x,0)},color=black,line width =1pt] (0pt,1pt) -- (0pt,-1pt);
\draw[color=black] (0,-1.5) -- (0,1.5);
\foreach \y in {-1,1}
\draw[shift={(0,\y)},color=black,line width =1pt] (1pt,0pt) -- (-1pt,0pt);
\draw[color=black,line width =0.25pt] (0pt,-10pt);
\clip(-1.6,-1.6) rectangle (1.6,1.6);
}

\newcommand{\drawinnerlines}{
\path [draw,line width =0.75pt] (0,0)--(1,1);
\path [draw,line width =0.75pt] (0,0)--(0,1);
\path [draw,line width =0.75pt] (0,0)--(-1,0);
\path [draw,line width =0.75pt] (0,0)--(1,-1);
\path [draw,line width =0.75pt] (0,0)--(2,1);
\path [draw,line width =0.75pt] (0,0)--(1,2);
\path [draw,line width =0.75pt] (0,0)--(-1,1);
\path [draw,line width =0.75pt] (0,0)--(0,-1);
\path [draw,line width =0.75pt] (0,0)--(1,0);
}

\newcommand{\drawouterlines}{
\path [draw,line width =0.75pt] (1,0)--(2,1);
\path [draw,line width =0.75pt] (2,1)--(1,1);
\path [draw,line width =0.75pt] (1,1)--(1,2);
\path [draw,line width =0.75pt] (1,2)--(0,1);
\path [draw,line width =0.75pt] (0,1)--(-1,1);
\path [draw,line width =0.75pt] (-1,1)--(-1,0);
\path [draw,line width =0.75pt] (-1,0)--(0,-1);
\path [draw,line width =0.75pt] (0,-1)--(1,-1);
\path [draw,line width =0.75pt] (1,-1)--(1,0);
}

\newcommand{\drawlines}{
\drawinnerlines
\drawouterlines
}
\newcommand{\drawdots}{
\fill [black] (1,0) circle (1.5pt);
\fill [black] (1,1) circle (1.5pt);
\fill [black] (0,1) circle (1.5pt);
\fill [black] (-1,0) circle (1.5pt);
\fill [black] (1,-1) circle (1.5pt);
\fill [black] (2,1) circle (1.5pt);
\fill [black] (1,2) circle (1.5pt);
\fill [black] (-1,1) circle (1.5pt);
\fill [black] (0,-1) circle (1.5pt);
}

\title[Piecewise linear  periodic maps]{Piecewise linear  periodic maps of the plane with integer coefficients}

\author{Grant Cairns}
\author{Yuri Nikolayevsky}
\author{Gavin Rossiter}

\address{Dept of Mathematics and Statistics, La Trobe University, Melbourne, Australia 3086}
\email{G.Cairns@latrobe.edu.au}
\email{Y.Nikolayevsky@latrobe.edu.au}
\email{gbrossiter@students.latrobe.edu.au}


\begin{document}

\maketitle

%\begin{abstract} ...
%\end{abstract}

\section{Introduction}

In 1993,
 Morton Brown published an interesting piecewise linear, periodic map of the plane with integer coefficients, stating that it had found applications in dynamics, topology and combinatorics  \cite{Mort2}. The Zentralblatt  review \cite{Ch} of the paper reported:
 \begin{quotation}
This short article contains two proofs of a mathematical gem: The homeomorphism of the plane given by $(x,y)\mapsto (\vert x\vert - y,x)$ is periodic with period~9!
\end{quotation}
 Brown had posed the problem of proving its periodicity 10 years earlier in \cite{Mort1}, wherein the problem was expressed in terms of the recurrence relation $h_{n+1}=|h_{n}|-h_{n-1}$.
 In \cite{Mort2}, Brown gave  his own  geometric proof and a combinatorial proof due to Donald Knuth. He quotes Knuth:
\begin{quotation}
When I saw advanced problem 6439, I couldn't believe that it was `advanced': a result like that has to be either false or elementary. But I soon found that it wasn't trivial. There is a simple proof, yet I can't figure out how on earth anybody would discover such a remarkable result. Nor have I discovered any similar recurrence having the same property. So in a sense I have no idea how to solve the problem properly. Is there an `insightful' proof, or is the result simply true by chance?
\end{quotation}
The aim of this paper is to revisit Brown's intriguing map. We give four other maps, of order 5, 7, 8 and 12, which are very similar to Brown's map in that they are piecewise linear in two pieces, the right half plane and the left half plane. We show that there are no other possible orders for such maps. We then show that every integer $n>1$ appears as the period of a piecewise linear, periodic map of the plane with integer coefficients, provided we allow enough piecewise linear pieces. Finally, we explain how the collection of these maps can be classified in terms of rooted binary tress of fixed height.

\section{Background remarks}

In its simplest incarnation, the crystallographic restriction says that if $A$ is a $2\times 2$ matrix with integer coefficients and $A^n=\id$ for some natural number $n$, then $n=1,2,3,4$ or $6$.
For more information and related results, see \cite{KP,BCK}.
Here is a proof of the crystallographic restriction (we will need similar ideas shortly). Let $A$ be such a matrix. As $A^n=\id$, we have $\det(A)=\pm1$. Similarly, if $A$ has a complex eigenvalue $\lambda$, with complex eigenvector $z$, then $A^nz=z$ gives $\lambda^n=1$, so $\lambda$ is a root of unity.
So if $\lambda$ is real, then $\lambda=\pm1$. In this case, $A^2$ has determinant 1, and it has 1 as an eigenvalue, so its Jordan canonical form is $\left(\begin{smallmatrix}1& 0\\w&1\end{smallmatrix}\right)$, for some $w$.  Then as $A^2$ has finite order, one obtains $w=0$ and thus $A^2=\id$. If $\lambda$ is not real, then as $\lambda$ satisfies the characteristic polynomial, $\lambda^2-\tr(A)\cdot \lambda+\det(A)=0$,  it has real part $\frac12\tr(A)$. In this case, as $\tr(A)$ is an integer and $\lambda$ is a root of unity, we have $\tr(A)=0,\pm1$.
In this case, as $A$ has two distinct eigenvalues, $A$ is diagonalizable over $\mathbb C$; if $\tr(A)=0$, then $\lambda=\pm i$ and $A$ has order 4. If $\tr(A)=\pm1$, then $\lambda=\frac12(\pm1\pm\sqrt3)$ and $A$ has order 3 or 6. This proves the crystallographic restriction.

Note that when $A$ has nonreal eigenvalues, they are complex conjugates, so their product, which is $\det(A)$, is positive. So $\det(A)=1$. Choosing an arbitrary nonzero vector $z$ in $\R^2$, we can then consider the basis $\{-Az,z\}$, relative to which  $A$ is conjugate to the matrix $\left(\begin{smallmatrix}\tr(A)& -1\\1&0\end{smallmatrix}\right)$.
This gives the following  examples of matrices of order 3,4,6 respectively:
\[
\alpha = \begin{pmatrix}
	  -1   & -1\\
	  1   & 0
	 \end{pmatrix},\qquad
	 \beta = \begin{pmatrix}
	 0   & -1\\
	 1   & 0
	 \end{pmatrix},\qquad \gamma = \begin{pmatrix}
	  1   & -1\\
	  1   & 0
	 \end{pmatrix}.
\]
The periodicity of these matrices has a nice geometric description. In Figure \ref{F:lin}, the respective matrices map  region $i$ to region $i+1$, modulo the order of the matrix.

\begin{figure}[h!]
\begin{tikzpicture}[scale=1.5]
\drawaxes
\fill[yellow](0,0)--(1,-1)--(0,1);
\fill[green](0,0)--(0,1)--(-1,0);
%\fill[red](0,0)--(1,-1)--(-1,0);
\fill[cyan,path fading =north](0,0)--(-1,0)--(1,-1)--cycle;
\path [draw,line width =0.75pt] (0,0)--(-1,0);
\path [draw,line width =0.75pt] (0,0)--(0,1);
\path [draw,line width =0.75pt] (0,0)--(1,-1);
\path [draw,line width =0.75pt] (-1,0)--(1,-1)--(0,1)--(-1,0);
\node[](1) at (-0.3,0.3) {1};
\node[](3) at (0.3,0) {0};
\node[](2) at (0,-0.25) {2};
\end{tikzpicture}
\begin{tikzpicture}[scale=1.5]
\drawaxes
\fill[green](0,0)--(1,0)--(0,-1);
\fill[yellow](0,0)--(0,-1)--(-1,0);
\fill[green](0,0)--(0,1)--(-1,0);
\fill[yellow](0,1)--(0,0)--(1,0);
\path [draw,line width =0.75pt] (0,0)--(1,0);
\path [draw,line width =0.75pt] (0,0)--(0,1);
\path [draw,line width =0.75pt] (0,0)--(-1,0);
\path [draw,line width =0.75pt] (0,0)--(0,-1);
\path [draw,line width =0.75pt] (1,0)--(0,-1)--(-1,0)--(0,1)--(1,0);
\node[](1) at (0.3,0.3) {0};
\node[](4) at (0.3,-0.3) {3};
\node[](3) at (-0.3,-0.3) {2};
\node[](2) at (-0.3,0.3) {1};
\end{tikzpicture}
\begin{tikzpicture}[scale=1.5]
\drawaxes
\fill[yellow](0,0)--(1,0)--(1,1);
\fill[green](0,0)--(0,1)--(1,1);
\fill[yellow](0,0)--(0,1)--(-1,0);
\fill[green](0,0)--(-1,0)--(-1,-1);
\fill[yellow](0,0)--(-1,-1)--(0,-1);
\fill[green](0,0)--(0,-1)--(1,0);
\path [draw,line width =0.75pt] (0,0)--(1,0);
\path [draw,line width =0.75pt] (0,0)--(1,1);
\path [draw,line width =0.75pt] (0,0)--(0,1);
\path [draw,line width =0.75pt] (0,0)--(-1,0);
\path [draw,line width =0.75pt] (0,0)--(-1,-1);
\path [draw,line width =0.75pt] (0,0)--(0,-1);
\path [draw,line width =0.75pt] (1,0)--(0,-1)--(-1,-1)--(-1,0)--(0,1)--(1,1)--(1,0);
\node[](1) at (0.7,0.3) {0};
\node[](6) at (0.3,0.7) {1};
\node[](5) at (-0.3,0.3) {2};
\node[](4) at (-0.7,-0.3) {3};
\node[](3) at (-0.3,-0.7) {4};
\node[](2) at (0.3,-0.3) {5};
\end{tikzpicture}
\caption{Maps of order 3,4,6}\label{F:lin}
\end{figure}
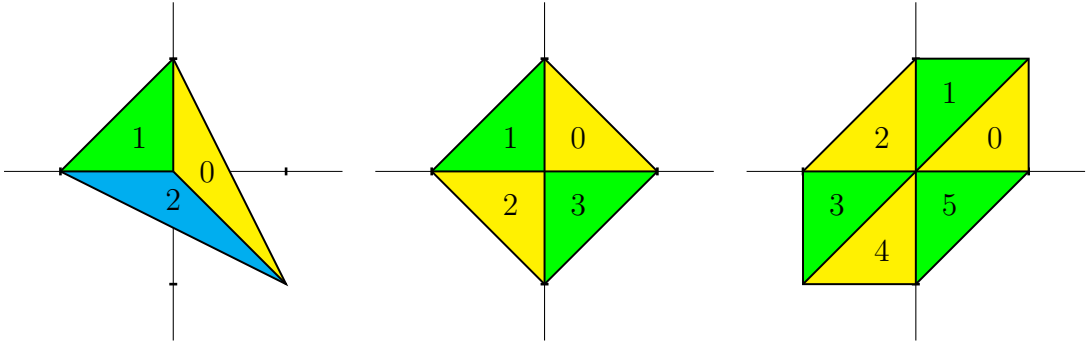

Brown's order 9 map,
  \begin{equation}\label{E:Bmap}
 H(x,y)= (\vert x\vert - y,x)
 \end{equation}
  is composed of two parts, one using the order 3 map $\alpha$ and the other using the order 6 map $\gamma$. It can be written explicitly as follows:
\[
H
\left(\begin{smallmatrix}x\\y\end{smallmatrix}\right) =
\begin{cases}
 \gamma \left(\begin{smallmatrix}x\\y\end{smallmatrix}\right) &:\ \text{if }x\ge0,  \\
 \alpha \left(\begin{smallmatrix}x\\y\end{smallmatrix}\right) &:\ \text{if }x<0.
\end{cases}
\]
The picture for Brown's map is shown in  Figure \ref{F:MB}. The periodicity of the map is established simply by verifying that the vertices of the polygon are mapped to each other in the fashion indicated in the figure. This is Brown's geometric proof.

%\newpage

\begin{figure}[h!]

\begin{tikzpicture}[scale=2.,line cap=round,line join=round,>=triangle 45,x=1.0cm,y=1.0cm]

%\tikz\fill[white] (0,3) circle (1.5pt);

\draw[color=black] (-2.5,0) -- (2.5,0);
\foreach \x in {-2,-1,1,2}
\draw[shift={(\x,0)},color=black] (0pt,2pt) -- (0pt,-2pt);
\draw[color=black] (0,-1.1) -- (0,2.1);
\foreach \y in {-1,1,2}
\draw[shift={(0,\y)},color=black] (2pt,0pt) -- (-2pt,0pt);
\draw[color=black] (0pt,-10pt);

\fill[yellow](0,0)--(2,1)--(1,0);
\fill[yellow](0,0)--(1,2)--(1,1);
\fill[yellow](0,0)--(0,1)--(-1,1);
\fill[yellow](0,0)--(-1,0)--(0,-1);
\fill[cyan,path fading =east](1,-1)--(0,0)--(1,0)--cycle;
\fill[green](0,0)--(1,1)--(2,1);
\fill[green](0,0)--(1,2)--(0,1);
\fill[green](0,0)--(-1,1)--(-1,0);
\fill[green](0,0)--(0,-1)--(1,-1);

\path [draw,line width =0.75pt] (0,0)--(-1,0);
\path [draw,line width =0.75pt] (0,0)--(0,-1);
\path [draw,line width =0.75pt] (0,0)--(1,-1);
\path [draw,line width =0.75pt] (0,0)--(1,0);
\path [draw,line width =0.75pt] (0,0)--(2,1);
\path [draw,line width =0.75pt] (0,0)--(1,1);
\path [draw,line width =0.75pt] (0,0)--(1,2);
\path [draw,line width =0.75pt] (0,0)--(0,1);
\path [draw,line width =0.75pt] (0,0)--(-1,1);
\path [draw,line width =0.75pt] (-1,0)--(0,-1)--(1,-1)--(1,0)--(2,1)--(1,1)--(1,2)--(0,1)--(-1,1)--(-1,0);

\begin{scriptsize}
\node[](1) at (1,0.25) {0};
\node[](2) at (0.75,1) {1};
\node[](3) at (-0.3,0.75) {2};
\node[](4) at (-0.3,-0.3) {3};
\node[](5) at (0.7,-0.25) {4};
\node[](6) at (1,0.7) {5};
\node[](7) at (0.25,0.9) {6};
\node[](8) at (-0.7,0.3) {7};
\node[](9) at (0.3,-0.7) {8};
\end{scriptsize}

\fill [black] (1,0) circle (1.5pt);
\fill [black] (2,1) circle (1.5pt);
\fill [black] (1,1) circle (1.5pt);
\fill [black] (1,2) circle (1.5pt);
\fill [black] (-1,1) circle (1.5pt);
\fill [black] (0,1) circle (1.5pt);
\fill [black] (0,-1) circle (1.5pt);
\fill [black] (-1,0) circle (1.5pt);
\fill [black] (1,-1) circle (1.5pt);
\end{tikzpicture}

\caption{Brown's map $H$}\label{F:MB}
\end{figure}
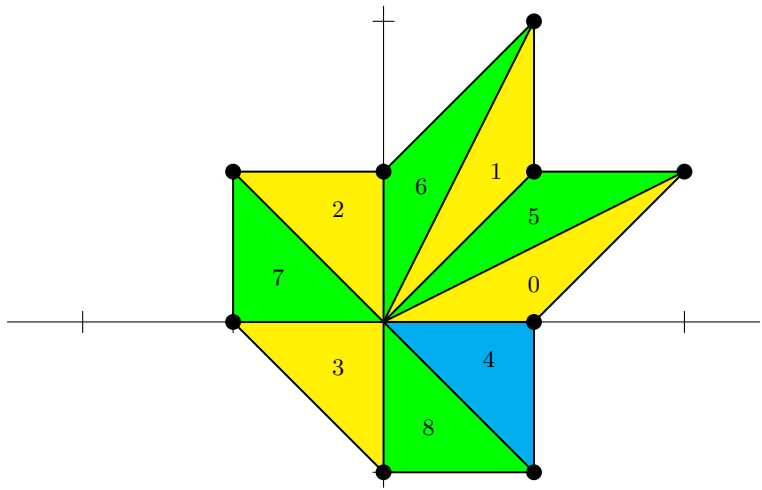

\section{Similar examples}

Brown's map combines the matrices $\alpha,\gamma$. Similar maps can be constructed by combining $\alpha,\beta$ and $\beta,\gamma$, as we now show.
First consider the map $G$ of the plane defined by
\[
G
\left(\begin{smallmatrix}x\\y\end{smallmatrix}\right) =
\begin{cases}
 \gamma \left(\begin{smallmatrix}x\\y\end{smallmatrix}\right) &:\ \text{if }x\ge0,  \\
 \beta\left(\begin{smallmatrix}x\\y\end{smallmatrix}\right) &:\ \text{if }x<0.
\end{cases}
\]
The map $G$ has order 5; the corresponding picture is shown in Figure \ref{F:o5}. Like Brown's map, $G$ can be written succinctly using the absolute value function: $G:(x,y)\mapsto (\frac{|x|+x}2-y,x)$.
And in the same way that Brown's map encodes the recurrence relation $h_{n+1}=|h_n|-h_{n-1}$, the map $G$ has a related recurrence relation $g_{n+1}=\frac{|g_n|+g_n}2-g_{n-1}$.

Now consider the map $F$  defined by
\[
F
\left(\begin{smallmatrix}x\\y\end{smallmatrix}\right) =
\begin{cases}
 \beta \left(\begin{smallmatrix}x\\y\end{smallmatrix}\right) &:\ \text{if }x\ge0,  \\
 \alpha\left(\begin{smallmatrix}x\\y\end{smallmatrix}\right) &:\ \text{if }x<0.
\end{cases}
\]
The map $F$ has order 7; the corresponding picture is shown in Figure \ref{F:o7}.  One has $F:(x,y)\mapsto (\frac{|x|-x}2-y,x)$, and $F$ has a related recurrence relation $f_{n+1}=\frac{|f_n|-f_n}2-f_{n-1}$.

\begin{figure}[h!]

\begin{minipage}{.4\textwidth}
  \centering
\begin{tikzpicture}[scale=2.]
\draw[color=black,line width =0.25pt] (-1.5,0) -- (1.5,0);
\draw[color=black] (0,-1.5) -- (0,1.5);
\draw[color=black,line width =0.25pt] (0pt,-10pt);

\fill[green](0,0)--(0,-1)--(-1,0);
\fill[yellow](0,0)--(0,1)--(-1,0);
\fill[yellow](0,0)--(1,1)--(1,0);
\fill[green](0,0)--(0,1)--(1,1);
%\fill[green](0,0)--(0,-1)--(1,0);
%\fill[yellow,path fading =east](0,-1)--(0,0)--(1,0)--cycle;
\fill[cyan,path fading =west](0,-1)--(0,0)--(1,0)--cycle;

\path [draw,line width =0.75pt] (0,0)--(-1,0);
\path [draw,line width =0.75pt] (0,0)--(0,-1);
\path [draw,line width =0.75pt] (0,0)--(1,0);
\path [draw,line width =0.75pt] (0,0)--(1,1);
\path [draw,line width =0.75pt] (0,0)--(0,1);
\path [draw,line width =0.75pt] (-1,0)--(0,-1)--(1,0)--(1,1)--(0,1)--(-1,0);
\node[](1) at (0.7,0.3) {0};
\node[](2) at (0.3,0.7) {1};
\node[](3) at (-0.3,0.3) {2};
\node[](4) at (-0.3,-0.3) {3};
\node[](5) at (0.3,-0.3) {4};

\fill [black] (-1,0) circle (1.5pt);
\fill [black] (0,-1) circle (1.5pt);
 \fill [black] (1,0) circle (1.5pt);
\fill [black] (1,1) circle (1.5pt);
\fill [black] (0,1) circle (1.5pt);
\end{tikzpicture}
\caption{The map $G$}\label{F:o5}
\end{minipage}
\hfill
\begin{minipage}{.4\textwidth}
  \centering
\begin{tikzpicture}[scale=2.]
\draw[color=black,line width =0.25pt] (-1.5,0) -- (1.5,0);
\draw[color=black] (0,-1.5) -- (0,1.5);
%\draw[color=black,line width =0.25pt] (0pt,-10pt);

\fill[yellow](0,0)--(1,0)--(1,1);
\fill[green](0,0)--(0,1)--(1,1);
\fill[yellow](0,0)--(0,1)--(-1,1);
\fill[green](0,0)--(-1,1)--(-1,0);
\fill[yellow](0,0)--(0,-1)--(-1,0);
\fill[green](0,0)--(1,-1)--(0,-1);
%\fill[green](1,-1)--(0,0)--(1,0)--cycle;
%\fill[yellow,path fading =east](1,-1)--(0,0)--(1,0)--cycle;
%\fill[yellow,path fading =north](1,-1)--(0,0)--(1,0)--cycle;
\fill[cyan,path fading =east](1,-1)--(0,0)--(1,0)--cycle;

\path [draw,line width =0.75pt] (0,0)--(1,0);
\path [draw,line width =0.75pt] (0,0)--(1,1);
\path [draw,line width =0.75pt] (0,0)--(0,1);
\path [draw,line width =0.75pt] (0,0)--(-1,1);
\path [draw,line width =0.75pt] (0,0)--(-1,0);
\path [draw,line width =0.75pt] (0,0)--(0,-1);
\path [draw,line width =0.75pt] (0,0)--(1,-1);
\path [draw,line width =0.75pt] (1,0)--(1,1)--(0,1)--(-1,1)--(-1,0)--(0,-1)--(1,-1)--(1,0);
\node[](1) at (0.7,0.3) {0};
\node[](2) at (-0.3,0.7) {1};
\node[](3) at (-0.3,-0.3) {2};
\node[](4) at (0.7,-0.3) {3};
\node[](5) at (0.3,0.7) {4};
\node[](6) at (-0.7,0.3) {5};
\node[](7) at (0.3,-0.7) {6};

\fill [black] (1,0) circle (1.5pt);
\fill [black] (1,1) circle (1.5pt);
 \fill [black] (0,1) circle (1.5pt);
\fill [black] (-1,1) circle (1.5pt);
\fill [black] (-1,0) circle (1.5pt);
\fill [black] (0,-1) circle (1.5pt);
\fill [black] (1,-1) circle (1.5pt);

\end{tikzpicture}
\caption{The map $F$}\label{F:o7}
\end{minipage}
\end{figure}

The maps we have seen so far have used the matrices $\alpha,\beta,\gamma$, which are all elliptic matrices; that is, their trace has absolute value $<2$. Some investigation reveals other piecewise linear maps involving parabolic matrices (i.e., those whose trace has absolute value 2) or hyperbolic matrices (i.e., those whose trace has absolute value $>2$). Let us introduce the matrices
\[
\mu = \begin{pmatrix}
	  -2   & -1\\
	  1   & 0
	 \end{pmatrix},\qquad
\nu = \begin{pmatrix}
	  -3   & -1\\
	  1   & 0
	 \end{pmatrix}.
\]
Now consider the maps $E,D$  defined by
\[
E
\left(\begin{smallmatrix}x\\y\end{smallmatrix}\right) =
\begin{cases}
 \alpha \left(\begin{smallmatrix}x\\y\end{smallmatrix}\right) &:\ \text{if }x\ge0,  \\
 \mu\left(\begin{smallmatrix}x\\y\end{smallmatrix}\right) &:\ \text{if }x<0,
\end{cases}\qquad
D
\left(\begin{smallmatrix}x\\y\end{smallmatrix}\right) =
\begin{cases}
 \alpha \left(\begin{smallmatrix}x\\y\end{smallmatrix}\right) &:\ \text{if }x\ge0,  \\
 \nu\left(\begin{smallmatrix}x\\y\end{smallmatrix}\right) &:\ \text{if }x<0.
\end{cases}
\]
The map $E$ has order 8; the corresponding picture is shown in Figure \ref{F:mu}.
The map $D$ has order 12; the corresponding picture is shown in Figure \ref{F:nu}.
Of course, these maps can also be expressed in formulas using the absolute value function, and they also encode second order recurrence relations.

\begin{figure}[h!]
\begin{tikzpicture}[scale=2.]
\draw[color=black,line width =0.25pt] (-1.5,0) -- (1.5,0);
\draw[color=black] (0,-1.5) -- (0,1.5);
\draw[color=black,line width =0.25pt] (0pt,-10pt);

\fill[yellow](0,0)--(0,1)--(1,0);
\fill[green](0,0)--(-1,1)--(-1,0);
\fill[yellow](0,0)--(1,-1)--(2,-1);
\fill[green](0,0)--(-1,2)--(0,1);
\fill[yellow](0,0)--(-1,0)--(0,-1);
\fill[green](0,0)--(2,-1)--(1,0);
\fill[yellow](0,0)--(-1,2)--(-1,1);
\fill[green](0,0)--(0,-1)--(1,-1);

\path [draw,line width =0.75pt] (0,0)--(1,0);
\path [draw,line width =0.75pt] (0,0)--(-1,0);
\path [draw,line width =0.75pt] (0,0)--(2,-1);
\path [draw,line width =0.75pt] (0,0)--(-1,2);
\path [draw,line width =0.75pt] (0,0)--(0,-1);
\path [draw,line width =0.75pt] (0,0)--(0,1);
\path [draw,line width =0.75pt] (0,0)--(-1,1);
\path [draw,line width =0.75pt] (0,0)--(1,-1);
\path [draw,line width =0.75pt] (2,-1)--(-1,2)--(-1,0)--(0,-1)--(2,-1);
\node[](1) at (0.3,0.3) {0};
\node[](2) at (-0.7,0.3) {1};
\node[](3) at (1,-0.7) {2};
\node[](4) at (-0.3,1) {3};
\node[](5) at (-0.3,-0.3) {4};
\node[](6) at (1,-0.3) {5};
\node[](7) at (-0.7,1) {6};
\node[](8) at (0.3,-.7) {7};

\fill [black] (0,1) circle (1.5pt);
\fill [black] (-1,0) circle (1.5pt);
 \fill [black] (2,-1) circle (1.5pt);
\fill [black] (-1,2) circle (1.5pt);
\fill [black] (0,-1) circle (1.5pt);
\fill [black] (1,0) circle (1.5pt);
\fill [black] (-1,1) circle (1.5pt);
\fill [black] (1,-1) circle (1.5pt);
\end{tikzpicture}
\caption{The map $E$}\label{F:mu}
\end{figure}
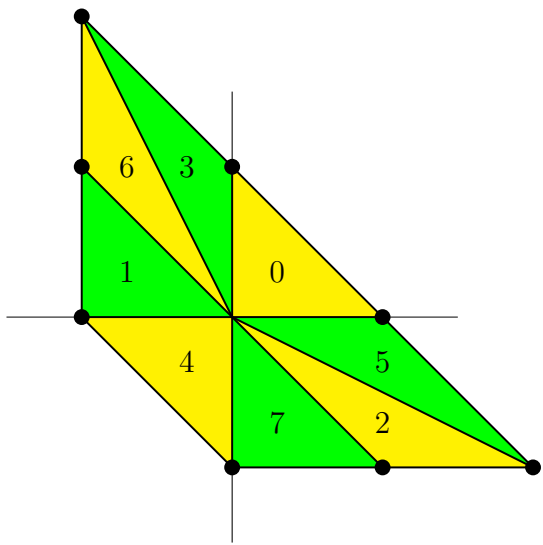

\begin{figure}[h!]
\begin{tikzpicture}[scale=2.]
\draw[color=black,line width =0.25pt] (-1.5,0) -- (1.5,0);
\draw[color=black] (0,-1.5) -- (0,1.5);
\draw[color=black,line width =0.25pt] (0pt,-10pt);

\fill[yellow](0,0)--(0,1)--(1,0);
\fill[green](0,0)--(-1,3)--(0,1);
\fill[yellow](0,0)--(-1,3)--(-1,2);
\fill[green](0,0)--(-1,2)--(-2,3);
\fill[yellow](0,0)--(-2,3)--(-1,1);
\fill[green](0,0)--(-1,1)--(-1,0);
\fill[yellow](0,0)--(-1,0)--(0,-1);
\fill[green](0,0)--(0,-1)--(1,-1);
\fill[yellow](0,0)--(1,-1)--(3,-2);
\fill[green](0,0)--(3,-2)--(2,-1);
\fill[yellow](0,0)--(3,-1)--(2,-1);
\fill[green](0,0)--(3,-1)--(1,0);

\path [draw,line width =0.75pt] (0,0)--(1,0);
\path [draw,line width =0.75pt] (0,0)--(-1,0);
\path [draw,line width =0.75pt] (0,0)--(3,-1);
\path [draw,line width =0.75pt] (0,0)--(-2,3);
\path [draw,line width =0.75pt] (0,0)--(3,-2);
\path [draw,line width =0.75pt] (0,0)--(-1,3);
\path [draw,line width =0.75pt] (0,0)--(0,-1);
\path [draw,line width =0.75pt] (0,0)--(0,1);
\path [draw,line width =0.75pt] (0,0)--(-1,1);
\path [draw,line width =0.75pt] (0,0)--(2,-1);
\path [draw,line width =0.75pt] (0,0)--(-1,2);
\path [draw,line width =0.75pt] (0,0)--(1,-1);
\path [draw,line width =0.75pt] (3,-1)--(1,0)--(0,1)--(-1,3)--(-1,2)--(-2,3)--(-1,1)--(-1,0)--(0,-1)--(1,-1)--(3,-2)--(2,-1)--(3,-1)--(1,0);
\node[](1) at (0.3,0.3) {0};
\node[](2) at (-0.7,0.3) {1};
\node[](3) at (2,-0.8) {2};
\node[](4) at (-1,1.7) {3};
\node[](5) at (1.25,-1) {4};
\node[](6) at (-0.2,1) {5};
\node[](7) at (-0.3,-.3) {6};
\node[](8) at (1.1,-0.2) {7};
\node[](9) at (-1,1.3) {8};
\node[](10) at (1.75,-1) {9};
\node[](11) at (-.75,1.8) {10};
\node[](11) at (.3,-.7) {11};

\fill [black] (0,1) circle (1.5pt);
\fill [black] (-1,0) circle (1.5pt);
 \fill [black] (3,-1) circle (1.5pt);
\fill [black] (-2,3) circle (1.5pt);
\fill [black] (3,-2) circle (1.5pt);
\fill [black] (-1,3) circle (1.5pt);
\fill [black] (0,-1) circle (1.5pt);
\fill [black] (1,0) circle (1.5pt);
\fill [black] (-1,1) circle (1.5pt);
\fill [black] (2,-1) circle (1.5pt);
\fill [black] (-1,2) circle (1.5pt);
\fill [black] (1,-1) circle (1.5pt);
\end{tikzpicture}
\caption{The map $D$}\label{F:nu}
\end{figure}
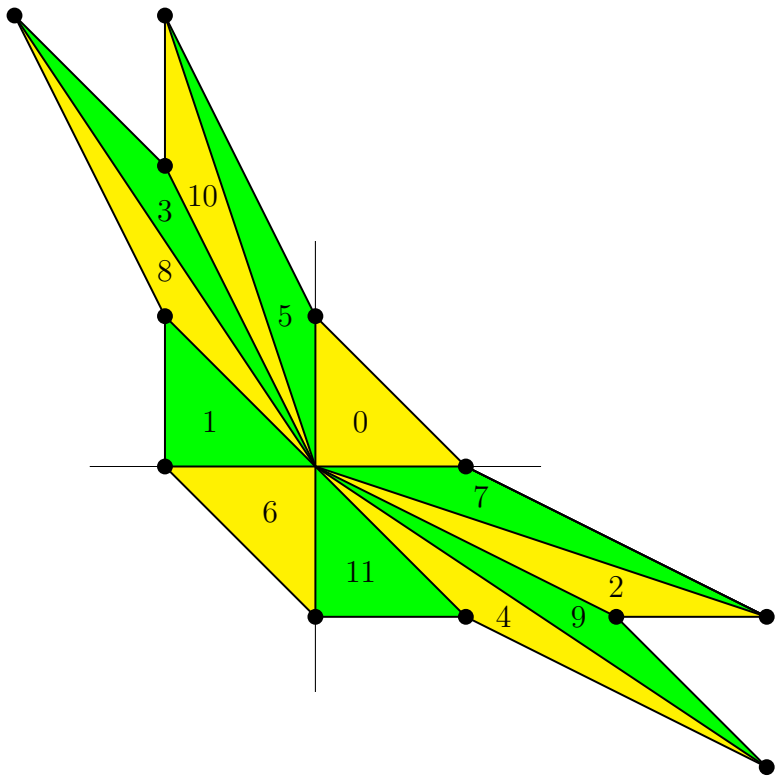

\section{Maps with half-plane piecewise linear pieces}

The aim of this section is to show that the maps we have seen so far effectively exhaust all the possible immediate generalisations of Brown's map.

\begin{theorem}\label{T:halfplanes}
Suppose that $f$ is a periodic continuous map of the plane that is linear with integer coefficients in each half plane $x\geq0$ and $x<0$. Then $f$ has period $1,2,3,4,5,6,7,8,9$ or $12$.\end{theorem}

Before presenting the proof of this theorem, we make a remark about general piecewise linear maps, not just maps that are linear in each half plane. We will use this remark in the proof of the theorem, and again later in the paper.

\begin{remark}\label{R:invol} Suppose that $f$ is a  piecewise linear map of the plane with period $n$ defined by matrices $A_1,\dots,A_k$ on respective open cones $C_1,\dots,C_k$. So the complement of the union of these cones is a union of closed rays $R_1,\dots , R_k$ say.
Consider the union $R$ of the images of these rays:
\[
R=\bigcup_{i=1}^k\bigcup_{j=1}^n f^j(R_i).
\]
 The set $R$ is a finite union of the rays, and the restriction of each power of $f$ to each connected component of the complement to $R$ is a linear map, which is a certain product of the $A_i$ (which may depend on the component). As these components are open sets, it follows that the corresponding matrix products for $f^n$ are all identity matrices. In particular, since the  $A_i$ are integer matrices, $\det A_i=\pm 1$ for all $i$. If $f$ is orientation preserving, then $\det A_i=1$ for each $i$, while if $f$ is orientation reversing, $\det A_i=-1$  for each $i$.

 Note that via the central projection, $f$ induces (and is determined by) a periodic homeomorphism of the unit circle. It follows from well known results (see \cite{CK}) that if $f$ is orientation reversing, then $f^2=\id$, while if  $f$ is orientation preserving, then $f$ is topologically conjugate to a rotation by the angle of $2\pi k /n$, for some $k$ and $n$ are coprime. There are infinitely many orientation reversing maps of period two. For example, consider the map defined by the matrix $A_1$ (resp.~$A_2$) in the right (resp.~left) half plane, where
\[
 A_1 =\left(\begin{smallmatrix}1&0\\0&-1\end{smallmatrix}\right) ,  \qquad
 A_2=\left(\begin{smallmatrix}1&0\\n&-1\end{smallmatrix}\right) .
\]
In the orientation preserving case,  the number $k/n$ is called the \emph{rotation number} of $f$. Such an $f$ has no fixed points on the unit circle.
In particular, none of the corresponding matrices $A_i$ has a positive eigenvalue.
 \end{remark}

\begin{proof}[Proof of Theorem \ref{T:halfplanes}] Suppose that $f$ is a piecewise linear map of period $n$ defined by two integer matrices $A_1, A_2$ say, with
\[
f
\left(\begin{smallmatrix}x\\y\end{smallmatrix}\right) =
\begin{cases}
 A_1 \left(\begin{smallmatrix}x\\y\end{smallmatrix}\right) &:\ \text{if }x\ge0,  \\
 A_2\left(\begin{smallmatrix}x\\y\end{smallmatrix}\right) &:\ \text{if }x<0.
\end{cases}
\]
From the above remark, we may assume that $\det A_1= \det A_2=1$.
The continuity of $f$ means that $A_1\left(\begin{smallmatrix}0\\1\end{smallmatrix}\right) =A_2\left(\begin{smallmatrix}0\\1\end{smallmatrix}\right)$; that is, $A_1,A_2$ have the same second column.
 By a global conjugation by a matrix of the form $\left(\begin{smallmatrix}1&0\\c&1\end{smallmatrix}\right)$ with $c\in\Z$, we may assume that
\[
 A_1 =\left(\begin{smallmatrix}a&-1\\1&0\end{smallmatrix}\right) ,  \qquad
 A_2=\left(\begin{smallmatrix}b&-1\\1&0\end{smallmatrix}\right),
\]
for some $a,b\in\Z$. From the above remark, neither $A_1$ nor $A_2$ has a positive eigenvalue.
So we have $a,b\leq 1$.  Furthermore, by a global conjugation by the matrix $-\id$ if necessary, we may assume that $a\geq b$. If $a=b$, then $f$ is itself linear, in which case its order is $1,2,3,4$ or $6$ by the crystallographic restriction. So we may assume that $1\geq a>b$.

If $b\geq -1$, there are only three possibilities:
\begin{enumerate}
\item $(a,b)=(1,-1)$, in which case $A_1=\gamma,A_2=\alpha$ and $f=H$.
\item $(a,b)=(1,0)$, in which case $A_1=\gamma,A_2=\beta$ and $f=G$.
\item $(a,b)=(0,-1)$, in which case $A_1=\beta,A_2=\alpha$ and $f=F$.
\end{enumerate}

We may assume that $a\geq -1$.
To see this, suppose that $a\leq -2$, so that $a,b\leq -2$. Consider the sequence $x_1,x_2,\dots$ defined by $(x_{m+1} ,x_m)^t= f^m (1, 0)^t$ for all $m\in\N$.
We will show by induction that $x_{m+1} /x_m <-1$ for all $m$. We have $x_{2} /x_1 =a<-1$. If $x_{m+1} /x_m <-1$, then
\[
\frac{x_{m+2}}{x_{m+1}}= \frac{c x_{m+1}- x_{m}}{x_{m+1}}=c-\frac{ x_{m}}{x_{m+1}}<c+1\leq -1,
\]
where $c=a$ or $b$, depending of $m$. But if $x_{m+1} /x_m <-1$ for all $m$, then $f$ cannot be periodic. So $a\in\{ -1,0,1\}$.

Note that there are no cases with $a=0$ and $b< -2$.
Indeed, suppose that $a=0$ and $b<  -2$.
The matrix $A_2$ has eigenvalues  $\lambda_{\pm}=\frac{b\pm\sqrt{b^2-4}}2$  and corresponding eigenvectors $X_{\pm}=(\lambda_{\pm},1)^t$, which lie in the second quadrant.
So $f(X_{\pm})=A_2X_{\pm}$ is in the fourth quadrant. Then $f^3(X_{\pm})=A_1^2A_2(X_{\pm})=-\lambda_{\pm} X_{\pm}$, since $A_1=\beta$, which is a rotation through $\pi/2$. But this is impossible as $|\lambda_{\pm}|\not=1$ and $f$ is periodic. When $a=0$ and $b= -2$, the matrix $A_2$ has the eigenvector $X=(-1,1)^t$ with eigenvalue $-1$, and the same argument shows that $f^3(X)=A_1^2A_2(X)=X$. So if $f$ is periodic it must have period 3. But  for example, $f^3(0,-1)^t=(-1,0)^t$. So the cases $a=0$ and $b< -1$ cannot appear.

Similarly, there are no cases with $a=1$ and $b< -2$. Indeed, suppose that $a=1$ and $b<  -2$ and let $X_{\pm}$ be as in the previous paragraph.
So as before, $f(X_{\pm})=A_2X_{\pm}$ is in the fourth quadrant. Then $f^4(X_{\pm})=A_1^3A_2(X_{\pm})=-\lambda_{\pm} X_{\pm}$, since $A_1=\gamma$ and $\gamma^3=-\id$  (see the right hand diagram in Figure \ref{F:lin}). But once again, this is impossible as $|\lambda_{\pm}|\not=1$ and $f$ is periodic. When $a=1$ and $b= -2$, the matrix $A_2$ has the eigenvector $X=(-1,1)^t$ and the same argument shows that $f^4(X)=A_1^3A_2(X)=X$. So if $f$ is periodic it must have period 4. But for example, $f^4(0,-1)^t=(-1,0)^t$. So the cases $a=1$ and $b< -1$ cannot appear.

When $a=-1,b=-2$, we have $f=E$ and when $a=-1,b=-3$, we have $f=D$, where $E$ and $D$ are the periodic maps defined at the end of the previous section.

It remains to show that the cases $a=-1,b<-3$ cannot occur.
Suppose that $a=-1,b\leq -4$ and let $\lambda= \frac{-b-2\pm\sqrt{b^2+4b}}2$ and consider the vector $X=(1,\frac{-1}{\lambda+1})$. Note that $\lambda>0$ and so $\frac{1}{\lambda+1}<1$. It follows that $f(X)=A_1(X)$ is in the second quadrant. So $f^2(X)=A_2A_1(X)$. But
\[
A_2A_1=\begin{pmatrix}-b-1&-b\\ -1&-1\end{pmatrix}.
\]
and $\lambda$ has been chosen so that it is an eigenvalue of $A_2A_1$ with eigenvector $X$. So $f^2(X)=\lambda X$. When $b<-4$, we have $\lambda>1$, and so $f$ cannot be periodic. When $b=-4$, we have $\lambda=1$ and so $f$ can only have periodic two. But $f^2(1,0)^t=A_2A_1(1,0)^t=(3,-1)^t$, so $b=-4$ is also impossible.
\end{proof}

\begin{remark}
Brown's map $H$ and the map $F$ of order 7 each wrap twice around the origin, so they have rotation number $\frac29$ and $\frac27$ respectively, while $\alpha,\beta,G,\gamma$ have rotation numbers $\frac13,\frac14,\frac15,\frac16$ respectively, and the maps $E,D$ have rotation numbers $\frac38,\frac5{12}$ respectively. The $5^{th}$ iterate $H^5$ has rotation number $\frac19$. It takes region $i$ to region $i+5 \pmod 9$  in Figure \ref{F:MB}. Similarly, the maps $E^3$ and $D^5$ have rotation numbers $\frac18,\frac1{12}$ respectively.
In general, if a map $f$ has rotation number $\frac{k}n$, where $k,n$ are coprime, then $f^j$ has rotation number $\frac{1}n$, where $j$ is the multiplicative inverse of $k$ in $\Z_n$.
Notice that the maps $H^5,E^3,D^5$ are not like the maps considered in this section. They are not composed of maps that are linear on two half planes. Instead they are piecewise linear with 9,8,12 piecewise linear parts respectively.
\end{remark}

\section{The general situation}

When one allows piecewise linear maps with more than two pieces, periods other than the ones we saw above are possible. In fact, any period is possible, as we will show below.

Consider $n$ distinct nonzero vectors $e_0,\dots,e_{n-1}$ in the plane, with integer coordinates, arranged in anti-clockwise order, so that the angle between each pair of successive vectors is less than $\pi$. Let $C_i$ denote the closed cone between the rays $\R^+e_i$ and $\R^+e_{i+1}$, where here and below, the subscripts are taken modulo $n$. Suppose we have a map $f$ of the plane such that for each $i\in\{0,1,\dots,n-1\}$, the restriction $f_i:=f|_{C_i}$ is defined by a matrix $A_i\in SL(2,\Z)$. Suppose furthermore that $f_i(e_{i-1})=e_{i}$, $f_i(e_i)=e_{i+1}$  and $f_i(C_{i-1})=C_{i}$ for all $i$. So $f$ is periodic with order $n$ and rotation number $\frac1n$.
Notice that since the matrices $A_i$ have determinant 1, the maps $f_i$ are area preserving. Thus, for each $i$, we have $e_{i-1}\times e_i=f_i(e_{i-1})\times f_i(e_i)=e_{i}\times e_{i+1}$, where $\times$ denotes the vector cross product. Hence $e_i\times (e_{i-1}+e_{i+1})=0$, and so $e_{i-1}+e_{i+1} =m_i e_i$ for some  $m_i\in\R$. Note that relative to the basis $\{e_{i-1},e_i\}$, the map $f_i$ has matrix representation $\left(\begin{smallmatrix}0&-1\\
1&m_i
\end{smallmatrix}\right) $. So, as $f_i$ has integer trace, we have $m_i\in\Z$.

 The sequence $m_0,m_1,\dots,m_{n-1}$ encodes much of the nature of the map. For the maps considered above, with the vertices of the polygons taken as our vectors $e_i$,  the corresponding sequences are shown in Table \ref{Table}.
Note that using the sequence $m_0,m_1,\dots,m_{n-1}$, the vectors $e_i$ can all be determined from $e_0,e_1$ by the recurrence relation $e_{i-1}+e_{i+1} =m_i e_i$, and the maps $f_1$ are also determined as their matrix representation $\left(\begin{smallmatrix}0&-1\\
1&m_i
\end{smallmatrix}\right) $ relative to the basis $\{e_{i-1},e_i\}$. So, up to a global linear conjugacy, the map is completely determined by the sequence. Indeed, as we will explain below, we can fix $e_0=(1,0)^t$ and $e_1=(0,1)^t$, and then the map is determined by the sequence.

\begin{table}[h!]
\begin{tabular}{|c|c|c|}\hline
map/matrix&order&sequences\\\hline
$\alpha$&3&$-1,-1,-1$\\\hline
$\beta$&4&$0,0,0,0$\\\hline
$G$&5&$0,1,1,1,0$\\\hline
$\gamma$&6&$1,1,1,1,1,1$\\\hline
$F^4$&7& $1,2,1,1,1,1,2$\\\hline
$E^3$&8&$2,1,2,1,1,2,1,2$ \\\hline
$C$&8&$1,2,1,2,1,2,1,2$\\\hline
$H^5$&9& $1,3,1,3,1,1,1,1,1,3$\\\hline
$D^5$&12& $3,1,3,1,3,1,1,3,1,3,1,3$\\\hline
\end{tabular}
\medskip
\caption{Maps and their sequences}\label{Table}
\end{table}

Not all integer sequences appear as the sequence of some piecewise linear map. For $n=3$, there is essentially only one map. Indeed, for $e_0=(1,0)^t$ and $e_1=(0,1)^t$, we have $m_2e_2=e_0+e_1=(1,1)^t$ and then
\begin{align*}
(m_0,0)^t&=m_0e_0=e_1+e_2=\left(\frac1{m_2},1+\frac1{m_2}\right)^t, \\
\text{and}\ (0,m_1)^t&=m_1e_1=e_0+e_2=\left(1+\frac1{m_2},\frac1{m_2}\right)^t.
\end{align*}
It follows that $m_0=m_1=m_2=-1$. So the map is (conjugate to) the globally linear map defined by the matrix $\alpha$. Notice that in this case,
the sum $e_0+e_1+e_2$ is zero. This observation will be useful later on.

For each $n>3$, there are infinitely many piecewise linear maps of order $n$, even up to global linear conjugacy. For $n=4$, examples are furnished by the polygons with vertices $e_0=(1,0)^t,e_1=(0,1)^t,e_2=(-1,0)^t,e_3=(m,-1)^t$, and matrices
\[
A_1=\begin{pmatrix}
0&-1\\
1&0
\end{pmatrix},\quad  A_2=\begin{pmatrix}
-m&-1\\
1&0
\end{pmatrix},\quad  A_3=\begin{pmatrix}
-m&-m^2-1\\
1&m
\end{pmatrix},\quad  A_0=\begin{pmatrix}
0&-1\\
1&m
\end{pmatrix}.  \]
The corresponding sequence is $0,-m,0,m$. The case $m=0$ is  the globally linear map determined by the matrix $\beta$.
The polygons for the values $m=1,2,3$ are shown in Figure \ref{F:o4}.

\begin{figure}[h!]
\begin{tikzpicture}[scale=1.4]

\fill[green](0,0)--(1,0)--(1,-1);
\fill[yellow](0,0)--(1,-1)--(-1,0);
\fill[green](0,0)--(0,1)--(-1,0);
\fill[yellow](0,1)--(0,0)--(1,0)--cycle;

\path [draw,line width =0.75pt] (0,0)--(1,-1);
\path [draw,line width =0.75pt] (0,0)--(1,0);
\path [draw,line width =0.75pt] (0,0)--(-1,0);
\path [draw,line width =0.75pt] (0,0)--(0,1);
\path [draw,line width =0.75pt] (1,0)--(1,-1)--(-1,0)--(0,1)--(1,0);
\node[](1) at (0.3,0.3) {0};
\node[](2) at (-0.3,0.3) {1};
\node[](3) at (-0.2,-0.2) {2};
\node[](4) at (0.7,-0.3) {3};

\fill [black] (1,0) circle (1.5pt);
\fill [black] (1,-1) circle (1.5pt);
 \fill [black] (-1,0) circle (1.5pt);
\fill [black] (0,1) circle (1.5pt);

\draw[color=black,line width =0.25pt] (-1.5,0) -- (1.5,0);
\draw[color=black] (0,-1) -- (0,1.5);
\draw[color=black,line width =0.25pt] (0pt,-10pt);

\end{tikzpicture}
\begin{tikzpicture}[scale=1.4]

\fill[green](0,0)--(1,0)--(2,-1);
\fill[yellow](0,0)--(2,-1)--(-1,0);
\fill[green](0,0)--(0,1)--(-1,0);
\fill[yellow](0,1)--(0,0)--(1,0)--cycle;

\path [draw,line width =0.75pt] (0,0)--(2,-1);
\path [draw,line width =0.75pt] (0,0)--(1,0);
\path [draw,line width =0.75pt] (0,0)--(-1,0);
\path [draw,line width =0.75pt] (0,0)--(0,1);
\path [draw,line width =0.75pt] (1,0)--(2,-1)--(-1,0)--(0,1)--(1,0);
\node[](1) at (0.3,0.3) {0};
\node[](2) at (-0.3,0.3) {1};
\node[](3) at (-0.1,-0.15) {2};
\node[](4) at (.8,-0.2) {3};

\fill [black] (1,0) circle (1.5pt);
\fill [black] (2,-1) circle (1.5pt);
 \fill [black] (-1,0) circle (1.5pt);
\fill [black] (0,1) circle (1.5pt);

\draw[color=black,line width =0.25pt] (-1.5,0) -- (1.5,0);
\draw[color=black] (0,-1) -- (0,1.5);
\draw[color=black,line width =0.25pt] (0pt,-10pt);

\end{tikzpicture}
\begin{tikzpicture}[scale=1.4]

\fill[green](0,0)--(1,0)--(3,-1);
\fill[yellow](0,0)--(3,-1)--(-1,0);
\fill[green](0,0)--(0,1)--(-1,0);
\fill[yellow](0,1)--(0,0)--(1,0)--cycle;

\path [draw,line width =0.75pt] (0,0)--(3,-1);
\path [draw,line width =0.75pt] (0,0)--(1,0);
\path [draw,line width =0.75pt] (0,0)--(-1,0);
\path [draw,line width =0.75pt] (0,0)--(0,1);
\path [draw,line width =0.75pt] (1,0)--(3,-1)--(-1,0)--(0,1)--(1,0);
\node[](1) at (0.3,0.3) {0};
\node[](2) at (-0.3,0.3) {1};
\node[](3) at (0.08,-0.15) {2};
\node[](4) at (1,-0.2) {3};

\fill [black] (1,0) circle (1.5pt);
\fill [black] (3,-1) circle (1.5pt);
 \fill [black] (-1,0) circle (1.5pt);
\fill [black] (0,1) circle (1.5pt);

\draw[color=black,line width =0.25pt] (-1.5,0) -- (2.5,0);
\draw[color=black] (0,-1) -- (0,1.5);
\draw[color=black,line width =0.25pt] (0pt,-10pt);

\end{tikzpicture}
\caption{Maps with sequences $\{0,-1,0,1\}$,  $\{0,-2,0,2\}$ and $\{0,-3,0,3\}$ respectively}\label{F:o4}
\end{figure}
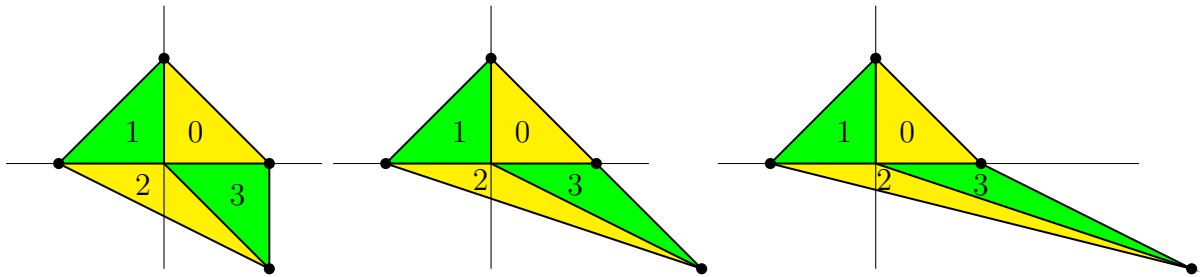

We will exhibit maps with $n>5$ below. But first we give a useful simplification, that will obviate the need to give explicit matrices $A_i$. The observation is two-fold. Firstly, notice that without loss of generality we may assume that the triangles $0,e_i,e_{i+1}$ all have area $\frac12$. Indeed, conjugating by an element of $SL(2,\Z)$ we may assume that $e_0=(x,0)^t$ and that  $e_1=(0,y)^t$ for some $x,y\in \N$. Since $e_1=A_1e_0$ and $A_1\in SL(2,\Z)$, we have $A_1=\left(\begin{smallmatrix}0&b\\
\frac{y}x&d
\end{smallmatrix}\right) $ for some $b,d \in \Z$, and so $y=x$ since $A$ has determinant $1$. Then using the recurrence relation $e_{i+1} =m_i e_i-e_{i-1}$, the coordinates of the vectors $e_i$ are all multiples of $x$. So we may replace the $e_i$ by $\frac1x e_{i}$. Then $e_0=(1,0)^t$ and $e_1=(0,1)^t$, so the triangle $0,e_0,e_1$ has area $\frac12$. Then as  the maps $f_i$ are area preserving, all the triangles $0,e_i,e_{i+1}$ have area $\frac12$.

\begin{definition}
Given a piecewise linear map $f$ of period $n$ and rotation number $\frac1n$, we will say that the polygon determined by the vertices $e_0,e_1,\dots,e_{n-1}$, as described above, is the \emph{fundamental polygon} of $f$.
\end{definition}

\begin{remark}
By Pick's theorem \cite{GS}, a closed planar triangle  $T$ having its vertices on the integer lattice has area $\frac12$ if and only if the only lattice
points in $T$ are its three vertices. So in a fundamental polygon of order $n$, the only lattice points are the origin and the $n$ vertices.
\end{remark}

The second observation, and this is the useful part, is that every star shaped polygon with  vertices $e_0,e_1,\dots,e_{n-1}$, listed counter-clockwise, with integer coordinates, for which the triangles $0,e_i,e_{i+1}$ all have area $\frac12$, occurs as the fundamental polygon of a piecewise linear map $f$. Indeed, we uniquely define $f$ by imposing the requirements that the linear restrictions $f_i$ satisfy $f_i(e_{i-1})=e_{i}$ and $f_i(e_i)=e_{i+1}$. This guarantees the continuity of $f$. It remains to see that the maps $f_i$ belong to $SL(2,\Z)$. Notice that the assumption that the triangles $0,e_i,e_{i+1}$ have area $\frac12$ means that the matrix $[e_i,e_{i+1}]$, whose columns are the vectors $e_i,e_{i+1}$, has determinant 1. So we can make a $SL(2,\Z)$ change of basis so that $e_i=(1,0)^t$ and $e_{i+1}=(0,1)^t$. Relative to this basis, $f_i$  has matrix representation $A_i=\left(\begin{smallmatrix}0&b_i\\
1&d_i
\end{smallmatrix}\right) $ for some $b_i,d_i\in\R$. But the second column of $A_i$ is the vector $A_ie_{i+1}$, which equals $e_{i+2}$, which has integer coordinates. So $b_i,d_i\in\Z$. The map is orientation preserving, so $b_i<0$. Finally, since $f^n=\id$, the product of the determinants of the $f_i$ is 1, so $b_i=-1$ for each $i$. So $A_i\in SL(2,\Z)$. (In the notation we used above, $d_i=m_{i+1}$).

The upshot of the above observations is that the map (and its sequence) is completely determined by its fundamental polygon, so to exhibit a piecewise linear map we just need to give its fundamental polygon. There is a simple construction that enables the construction of maps. Suppose one has a fundamental polygon with vertices $e_0,e_1,\dots,e_{n-1}$. Choose $i\in\{0,1,\dots,n-1\}$ and set $\hat e_i=e_i+e_{i+1}$. We have $\det[e_{i},\hat e_i]=\det[e_{i},e_i+e_{i+1}]=\det[e_{i},e_{i+1}]=1$, and similarly, $\det[\hat e_{i},e_{i+1}]=1$. So the vectors $e_0,e_1,\dots,e_i,\hat e_i,e_{i+1},\dots,e_{n-1}$ are the vertices of a fundamental polygon of a piecewise linear map with period $n+1$.
We will call this process \emph{vertex insertion}. Repeating employing vertex insertion, we can construct piecewise linear map of arbitrary order.

Let us now describe how vertex insertion can be reversed. Consider  a fundamental polygon $P$ with vertices $e_0,e_1,\dots,e_{n-1}$.  Suppose that for some $i\in\{0,1,\dots,n-1\}$, the vertex $e_i$ is a vertex of the convex hull of $P$ and furthermore, the angle at the origin formed by the vertices $e_{i-1},e_{i+1}$ is strictly less than $\pi$. So the vertices $0,e_{i-1},e_i,e_{i+1}$ form a convex quadrilateral $Q$.
It follows that $e_i=e_{i-1}+e_{i+1}$. To see this first note that as the triangle $0,e_{i-1},e_{i}$ has area $\frac12$, we can apply a $SL(2,\Z)$ transformation mapping $e_{i-1}$ to
$(1,0)^t$ and $e_{i}$ to $(0,1)^t$.
Then as the triangle $0,e_i,e_{i+1}$ has area $\frac12$, the vertex $e_{i+1}$ must have $x$-coordinate $-1$. Then as the quadrilateral $Q$ is convex, we have $e_{i+1}=(-1,1)^t$. So $e_i=e_{i-1}+e_{i+1}$ as claimed. Now remove vertex $e_i$. The triangle $0,e_{i-1},e_{i+1}$ has area $\frac12$. So the vertices $e_0,e_1,\dots,e_{i-1},e_{i+1},\dots,e_{n-1}$ form a fundamental polygon for a map of period $n-1$. We will call this process \emph{vertex removal}.

A little thought should convince the reader that starting with an given fundamental polygon, we can perform the vertex removal process repeatedly until we obtain a fundamental  polygon with just 4 vertices. We claim that for a fundamental  polygon $P$ with vertices $e_0,e_1,e_2,e_3$, at least two of the vectors are collinear; say $e_2=-e_0$. To see this note that if no two of them are collinear then we could make a further vertex removal and obtain the fundamental  polygon corresponding to the matrix $\alpha$ of order 3. But as we saw above, the sum of the vertex vectors for $\alpha$ is zero. It follows that when we re-construct $P$ by applying vertex insertion to $\alpha$, the inserted vertex is equal to minus the other vertex. So we may assume $e_2=-e_0$, as claimed.
Thus,  applying a conjugacy, we have the situation we considered above: $e_0=(1,0)^t,e_1=(0,1)^t,e_2=(-1,0)^t,e_3=(m,-1)^t$, and $m\geq 0$. Let us call the corresponding map $\varphi_m$. When $m=0$, the map $\varphi_0$ is just the map given by the matrix $\beta$; no further vertex removal is possible.
When $m=1$, we can make a further vertex removal by removing $e_0$, and again obtain the fundamental  polygon of $\alpha$.  When $m>1$,  no further vertex removal is possible.

\begin{theorem}\label{T:vi} Every piecewise linear map of the plane of order $n$ with rotation number $\frac1n$ is obtained by repeated vertex insertion from $\alpha,\beta$ or $\varphi_m$ for $m>1$.
\end{theorem}

\section{Admissible sequences}

We have seen that up to conjugacy, a piecewise linear map is  determined by its fundamental polygon $e_0,e_1,\dots,e_{n-1}$, or equivalently, by its
sequence $m_0,m_1,\dots,m_{n-1}$, where $e_{i-1}+e_{i+1} =m_i e_i$.  Let us say that a sequence of integers is \emph{admissible} if it appear as the sequence of some map. As we saw in the previous section, the only admissible sequence of length 3 is $-1,-1,-1$.

Note that if we insert a vertex $\hat e_i=e_i+e_{i+1}$ between $e_i$ and $e_{i+1}$ in a fundamental polygon $e_0,e_1,\dots,e_{n-1}$, then $e_{i-1}+\hat e_i=(m_i+1)e_i$ and $\hat e_i+e_{i+2}=(m_{i+1}+1)e_{i+1}$. Thus the sequence $m_0,m_1,\dots,m_{n-1}$ changes to $m_0,m_1,\dots,m_i+1,1,m_{i+1}+1,\dots,m_{n-1}$. For example, applying edge insertion to the sequence $-1,-1,-1$ gives the sequence $0,-1,0,1$.

By Theorem \ref{T:vi}, the admissible sequences of length $\geq 4$ are the sequences that can be obtained by vertex insertion, as described in the previous paragraph, starting with the sequence $0,-m,0,m$, where $m\geq 0$.
Notice that for the sequence $-1,-1,-1$ and the sequences $0,-m,0,m$, the sequence sum is $3n-12$, where $n$ is the length of the sequence. Since vertex insertion increases the length by 1 and the sequence sum by 3, we have that all admissible sequences of length $n$ have sum $3n-12$.
 In a sequence $m_0,m_1,\dots,m_{n-1}$, the entry $m_i$ is the trace of the map $f_i$, as described in the previous section. So $f_i$ is elliptic (resp.~parabolic, resp.~hyperbolic) if $|m_i|<2$ (resp.~$|m_i|=2$, resp.~$|m_i|>2$). The sequence sum, $\sum_i m_i$,  is the sum of the traces of the maps $f_i$. So the average trace is $3-\frac{12}n$. Thus Theorem \ref{T:vi} has the following corollary.

 \begin{corollary}
Consider a piecewise linear map of the plane  of period $n$ and rotation number $\frac1n$, composed of linear parts $f_i$ for $i=0,\dots,n-1$.
Then the average trace of the $f_i$ is $3-\frac{12}n$.
In particular, as $n$ tends to infinity the average of the traces  tends to $3$.\end{corollary}

Although vertex insertion gives an algorithm for producing all admissible sequences, there is no simple formula giving the admissible sequences of a particular length. The combinatorial difficulties involved should become apparent in the next section.

\section{Fundamental  polygons}

Because  fundamental  polygons with 4 vertices have two collinear vectors, and since the fundamental  polygons of higher order are built by vertex insertion, it follows that all  fundamental  polygons with at least 4 vertices have  two collinear vectors, which we may take to be $e_0=(1,0)^t$ and $e_k=(-1,0)^t$. So in order to describe all  fundamental polygons $P$ it suffices to describe the intersection of $P$ with the upper and lower half planes. Such a half plane  is formed by $k+1$ distinct nonzero vectors $e_0,\dots,e_{k}=-e_0$ arranged in the plane in anti-clockwise order, with integer coordinates with nonnegative $x$-coordinate, such that the triangles $0,e_i,e_{i+1}$ all have area $\frac12$. But as we explained above, up to an integer shear in the $x$-direction,  we may take $e_0=(1,0)^t$ and $e_1=(0,1)^t$. So we actually only need to describe the $k-2$ vectors $e_2,\dots,e_{k-1}$ of a fundamental  polygon in the second quadrant. These structures are in a one to one correspondence with rooted binary trees having $k-2$ nodes. To avoid making this too notationally cumbersome, we will simply describe this correspondence with an example.

As we have described in the previous section, the fundamental polygon is given by a number of vertex insertions. Each vertex insertion opens up the possibility of two further vertex insertions. This produces a binary tree whose nodes are labelled by the vertices that are inserted. Figure \ref{F:tree} shows the complete tree of possible insertions of height 4. A particular fundamental polygon is given by a subtree commencing at the root. For example, the fundamental polygon shown in Figure~\ref{F:up} corresponds to the tree shown in Figure \ref{F:treeEx}. Note that the tree appears naturally inside the fundamental polygon, as depicted in Figure~\ref{F:up}.

\begin{figure}
\begin{tikzpicture}[->,>=stealth',level/.style={sibling distance = 8.2cm/#1,
  level distance = 1.5cm}]
\node   {$(-1,1)$}
    child{ node  {$(-2,1)$}
            child{ node  {$(-3,1)$}
            	child{ node  {$(-4,1)$}}
							child{ node  {$(-4,2)$}}
            }
            child{ node  {$(-3,2)$}
							child{ node  {$(-5,3)$}}
							child{ node  {$(-4,3)$}}
            }
    }
    child{ node  {$(-1,2)$}
            child{ node  {$(-2,3)$}
							child{ node  {$(-3,4)$}}
							child{ node  {$(-3,5)$}}
            }
            child{ node {$(-1,3)$}
							child{ node  {$(-2,5)$}}
							child{ node  {$(-1,4)$}}
            }
		}
;
\end{tikzpicture}
\caption{The vertex insertion tree of height 3}\label{F:tree}
\end{figure}
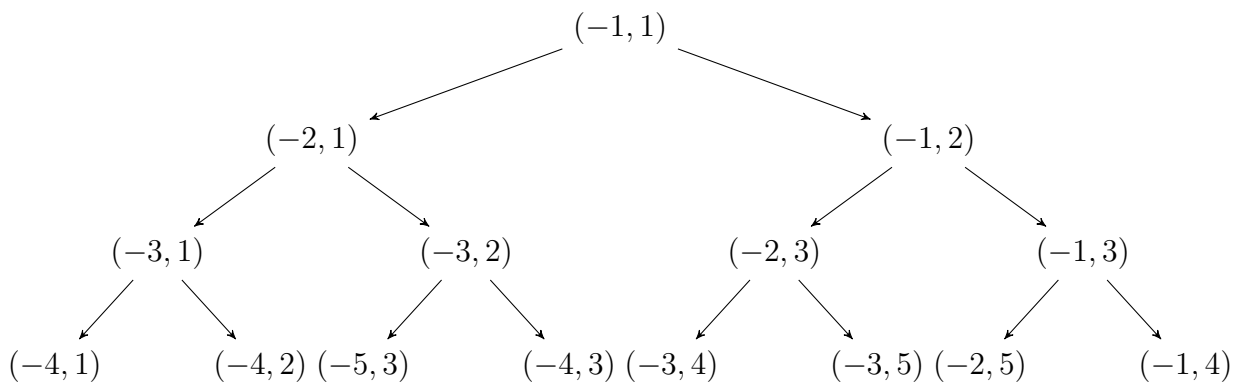

\begin{figure}[h!]
\begin{tikzpicture}[scale=2.]

\fill[yellow](0,0)--(1,0)--(0,1);
\fill[green](0,0)--(0,1)--(-1,2);
\fill[yellow](0,0)--(-1,2)--(-1,1);
\fill[green](0,0)--(-1,1)--(-3,2);
\fill[yellow](0,0)--(-2,1)--(-3,2);
\fill[green](0,0)--(-3,1)--(-2,1);
\fill[yellow](0,0)--(-3,1)--(-4,1);
\fill[green](0,0)--(-1,0)--(-4,1);

\path [draw,line width =0.25pt] (1,0)--(0,1)--(-1,2)--(-1,1)--(-3,2)--(-2,1)--(-3,1)--(-4,1)--(-1,0)--(1,0);
\path [draw,line width =0.25pt] (0,0)--(1,0);
\path [draw,line width =0.25pt] (0,0)--(-1,0);
\path [draw,line width =.25pt] (0,0)--(-1,2);
\path [draw,blue,line width =2pt] (0,0)--(-1,1);
\path [draw,blue,line width =2pt] (-1,2)--(-1,1);
\path [draw,blue,dashed,line width =2pt] (-2,1)--(-1,1);
\path [draw,blue,line width =2pt] (-2,1)--(-3,2);
\path [draw,blue,line width =2pt] (-2,1)--(-3,1);
\path [draw,blue,line width =2pt] (-4,1)--(-3,1);
\path [draw,line width =0.25pt] (0,0)--(-2,1);
\path [draw,line width =0.25pt] (0,0)--(-3,1);
\path [draw,line width =0.25pt] (0,0)--(-4,1);
\path [draw,line width =0.25pt] (0,0)--(-3,2);
\path [draw,line width =0.25pt] (0,0)--(0,1);

\fill [black] (1,0) circle (1.5pt);
 \fill [black] (-1,0) circle (1.5pt);
 \fill [black] (-1,2) circle (1.5pt);
 \fill [black] (-1,1) circle (1.5pt);
 \fill [black] (-2,1) circle (1.5pt);
 \fill [black] (-3,1) circle (1.5pt);
 \fill [black] (-4,1) circle (1.5pt);
 \fill [black] (-3,2) circle (1.5pt);
\fill [black] (0,1) circle (1.5pt);

\draw[color=black,line width =0.25pt] (-3.5,0) -- (1.5,0);
\draw[color=black] (0,-.3) -- (0,2.2);
\draw[color=black,line width =0.25pt] (0pt,-10pt);

\end{tikzpicture}
\caption{A fundamental polygon in the upper half plane}\label{F:up}
\end{figure}
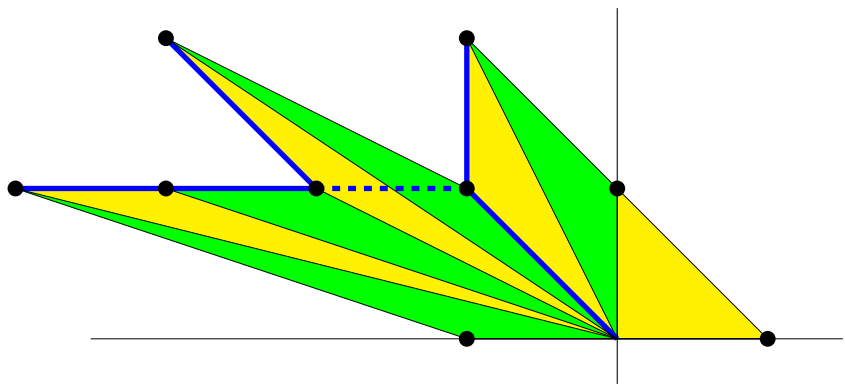

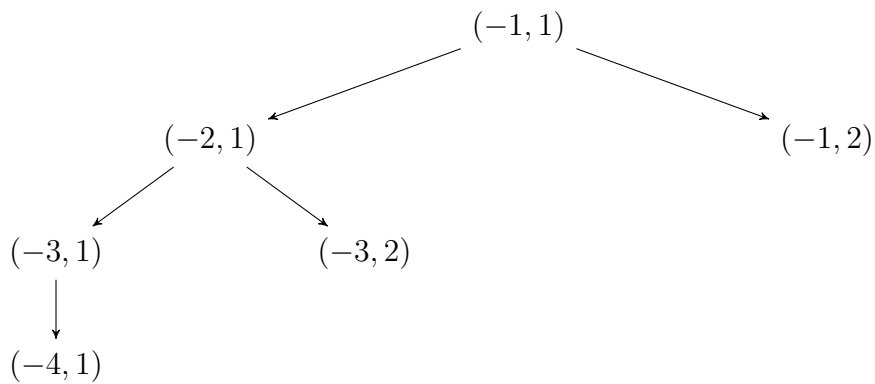
\begin{figure}
\begin{tikzpicture}[->,>=stealth',level/.style={sibling distance = 8.2cm/#1,
  level distance = 1.5cm}]
\node   {$(-1,1)$}
    child{ node  {$(-2,1)$}
            child{ node  {$(-3,1)$}
            	child{ node  {$(-4,1)$}}
            }
            child{ node  {$(-3,2)$}
            }
    }
    child{ node  {$(-1,2)$}
            	}
;
\end{tikzpicture}
\caption{The binary tree of Figure \ref{F:up}}\label{F:treeEx}
\end{figure}

A general fundamental polygon is determined, up to a global linear conjugation, by two rooted binary trees and an integer: one rooted binary tree for the fundamental polygon in the upper half plane, one rooted binary tree for the fundamental polygon in the lower half plane, and an integer $m$ determining a horizontal shear $\left(\begin{smallmatrix}1&m\\
0&1
\end{smallmatrix}\right) $ of the lower half plane. Finally, to uniquely determine the fundamental polygon, one needs to take account of the fact that one fundamental polygon might be obtained from another by the map $(x,y)\mapsto (-x,-y)$.

\bibliographystyle{amsplain}
\bibliography{mb}
 \end{document}